\documentclass[12pt,a4paper]{amsart}
\usepackage[top=1.4in,left=1.4in,right=1.4in,bottom=1.4in]{geometry}
\usepackage{amsmath,amssymb,amsthm,mathdots}
\usepackage{booktabs,enumerate,graphicx}

\usepackage{color}
\newcommand{\RED}[1]{{\color{red}#1}}
\renewcommand{\RED}[1]{#1} 

\newtheorem{theorem}{Theorem}
\newtheorem{proposition}[theorem]{Proposition}
\newtheorem{lemma}[theorem]{Lemma}
\theoremstyle{definition}
\newtheorem{definition}[theorem]{Definition}
\newtheorem{example}[theorem]{Example}
\newtheorem{remark}[theorem]{Remark}

\DeclareMathOperator{\Ann}{Ann}

\DeclareMathOperator{\Aut}{Aut}
\DeclareMathOperator{\Mat}{Mat}
\DeclareMathOperator{\Sym}{Sym}
\DeclareMathOperator{\Alt}{Alt}
\DeclareMathOperator{\Hilb}{Hilb}
\DeclareMathOperator{\Pf}{Pf}
\DeclareMathOperator{\Ad}{Ad}
\DeclareMathOperator{\ad}{ad}

\newcommand{\CC}{\mathbb{C}}
\newcommand{\frakgl}{\mathfrak{gl}}
\newcommand{\fraksp}{\mathfrak{sp}}
\newcommand{\fraksl}{\mathfrak{sl}}
\newcommand{\frako}{\mathfrak{o}}
\newcommand{\frakg}{\mathfrak{g}}
\newcommand{\frakp}{\mathfrak{p}}
\newcommand{\fraknp}{\mathfrak{n}^+}
\newcommand{\fraknm}{\mathfrak{n}^-}
\newcommand{\frakk}{\mathfrak{k}}
\newcommand{\frakh}{\mathfrak{h}}

\newcommand{\T}{\,{}^t\!}
\newcommand{\of}{\overline{f}}
\newcommand{\ol}{\overline{\lambda}}

\newcommand{\CDOTS}{
\put(0,0){\circle*{0.5}}%
\put(2.5,0){\circle*{0.5}}%
\put(-2.5,0){\circle*{0.5}}}
\newcommand{\CDOTSLINE}[4]{
\CDOTS%
\Line(5,0)(10,0)%
\Line(-5,0)(-10,0)}

\title[SLP of Gorenstein Algebras Generated by Relative Invariants]%
{The Strong Lefschetz Property of Gorenstein Algebras Generated by Relative Invariants}

\author{Takahiro Nagaoka}
\author{Akihito Wachi}
\address{
{\rm (Akihito Wachi)}
Department of Mathematics,
Hokkaido University of Education,
Kushiro, 085-8580, Japan.}



\begin{document}
\maketitle

\begin{abstract}
We prove the strong Lefschetz property for 
Artinian Gorenstein algebras generated by the relative invariants
of prehomogeneous vector spaces of commutative parabolic type.
\end{abstract}

\section{Introduction} 
Let $Q = k[x_1,x_2,\ldots,x_n]$
be a polynomial ring over a field $k$
\RED{of characteristic zero},
and $R= k[\partial/\partial x_1, \partial/\partial x_2, \ldots,
  \partial/\partial x_n]$
the ring of partial differential operators with constant coefficients.
Set $\Ann_R(F) = \{ P \in R \mid P(F) = 0 \}$ for $F \in Q$.
Then $R / \Ann_R(F)$ is a graded Artinian Gorenstein algebra.
Conversely, any graded Artinian Gorenstein algebra is constructed this way
(Macaulay's dual annihilator theorem.
See \cite[Theorem~2.71]{MR3112920}).
$R / \Ann_R(F)$ is called a {\em Gorenstein algebra generated by $F$}.

\RED{%
The graded Artinian Gorenstein algebra
$R / \Ann_R(F) = \bigoplus_{i=0}^c A_i$
is a Poincar\'e duality algebra,
which is an algebra such that the map
$A_i \times A_{c-i} \to A_c \; (\simeq k)$ ($(a,b) \mapsto ab$)
forms a perfect paring for any $i = 0, 1, \ldots, \lfloor c/2 \rfloor$.
Therefore, graded Artinian Gorenstein algebras come as cohomology rings
in certain categories.
}%
\RED{%
The following definition is an algebraic abstraction of the hard Lefschetz theorem
for the cohomology rings of compact K\"ahler manifolds.
}%

\begin{definition}[Strong Lefschetz property]
A graded Artinian algebra $A = \bigoplus_{i=0}^c A_i$ 
($A_0 \simeq k$, $A_c \ne 0$)
over $k$
is said to have the 
{\em strong Lefschetz property} 
if there exists $L \in A_1$ such that
$\times L^{c-2i}: A_i \to A_{c-i}$ is bijective
for every $i = 0, 1, \ldots, \lfloor c/2 \rfloor$.
In this case $L$ is called a {\em Lefschetz element}.
\end{definition}

When a graded Artinian algebra $A$ is a complete intersection,
which is a special case of Gorenstein algebras,
there is a long-standing conjecture that
$A$ has the strong Lefschetz property.
In general,
to check whether $A$ has the strong Lefschetz property or not
is difficult,
and therefore
to check for which $F$ the algebra $R / \Ann_R(F)$
has the strong Lefschetz property is also difficult.

There are related studies on this question
\RED{%
in terms of the Hessian of $F$.
By Maeno-Watanabe \cite{MR2594646}
the multiplication map
$\times L^{c-2}: (R/\Ann_R(F))_1 \to (R/\Ann_R(F))_{c-1}$
is bijective for some $L \in R_1$
if and only if the Hessian of $F$ is not identically zero.
Moreover,
$R/\Ann_R(F)$ has the strong Lefschetz property
if and only if every `higher Hessian' of $F$ is not identically zero
\cite[Theorem~3.1]{MR2594646}.

Using this criterion
Maeno-Numata \cite{MR3566530} proved the strong Lefschetz property
when $F$ is the basis generating function of a matroid
if the lattice of flats of the matroid is modular geometric.
They also conjectured the strong Lefschetz property for any matroid.

Nagaoka-Yazawa \cite{MR4138666, MR4234203}
proved the bijectivity of
$\times L^{c-2}: (R/\Ann_R(F))_1 \to (R/\Ann_R(F))_{c-1}$
for some $L \in R_1$
when $F$ is the Kirchhoff polynomial of any simple graph.
In particular,
if $F$ is the Kirchhoff polynomial of a complete graph,
then $F$ is the determinant of a symmetric matrix,
which is
the case of $({\rm C}_n, n)$ in Table~\ref{tbl:comm-parab-type}.
For this reason this paper is motivated by their work.

Murai-Nagaoka-Yazawa \cite{MR4223331}
generalized this result to the case
when $F$ is the basis generating function of any matroid
(the basis generating function of a graphic matroid is
same as the Kirchhoff polynomial of a graph).


From a viewpoint of prehomogeneous vector spaces
the Hessian of a relative invariant is not identically zero
if and only if 
the prehomogeneous vector spaces is regular
(Sato-Kimura \cite[Definition~7]{MR430336}).
}%


As another approach to this question,
Gondim-Russo \cite{MR3282110}, 
Gondim \cite{MR3686979} 
and
Gondim-Zappal\`a \cite{MR3750213,MR3958091} 
study polynomials whose Hessians are identically zero.

In this paper
we give a new family of polynomials $F$ such that $R / \Ann_R(F)$ has the
strong Lefschetz property.
The family consists of the relative invariants of 
regular prehomogeneous vector spaces of commutative parabolic type
(see Definition~\ref{def:pv-parab}).
This family contains
determinants,
determinants of symmetric matrices,
Pfaffians of alternating matrices of even size,
$x_1^2 + x_2^2 + \cdots + x_n^2$
and a polynomial of degree three in 27 variables.
The family also contains powers of the above polynomials.

This paper is organized as follows.
In Section~\ref{sec:pv-parab}
we review the definition of prehomogeneous vector spaces of
commutative parabolic type.
In Section~\ref{sec:main-thm}
we state our main theorem and prove it in Section~\ref{sec:proof}.

\textbf{Acknowledgment.}
The authors thank Masatoshi Kitagawa for a critical comment on the proof
made in a seminar given by the first author,
which is at Waseda University in December, 2019.
The authors also thank Satoshi Murai
for advice on how we write the paper.
This work was supported by JSPS KAKENHI Grant Numbers 22K03347, 20K03508.

\section{Prehomogeneous vector spaces of commutative parabolic type} 
\label{sec:pv-parab}
In this section
we review the definition of prehomogeneous vector spaces
of commutative parabolic type.
Our main theorem shows that
their relative invariants $F$ give
Artinian Gorenstein algebras $R / \Ann_R(F)$ 
that have the strong Lefschetz property.

\begin{definition}[Prehomogeneous vector space]
For simplicity
we define prehomogeneous vector spaces 
over the complex number field.

Let $G$ be a complex Lie group,
and $(G, \pi, V)$ a representation of $G$ on a $\CC$-vector space $V$.
$(G, \pi, V)$ is called a 
{\em prehomogeneous vector space}
if there exists a Zariski open $G$-orbit on $V$.

A polynomial function $F$ on $V$
is called a {\em relative invariant} of $(G, \pi, V)$
if there exists a group character $\chi: G \to \CC^\times$
such that
$F(gv) = \chi(g) F(v)$ ($g \in G$, $v \in V$).
\end{definition}

\RED{%
\begin{definition}[Prehomogeneous vector space of commutative parabolic type]
\label{def:pv-parab}
Let $\frakg$ be a complex simple Lie algebra,
$\frakp$ its parabolic subalgebra,
and $\fraknp$ the nilpotent radical of $\frakp$.
Let $\frakk$ be a Levi subalgebra of $\frakp$,
which is, by definition, 
a subalgebra that is a complement to $\fraknp$ in $\frakp$,
and $K$ the complex Lie subgroup of $G$
corresponding to $\frakk$.
Let $\frakh$ be a Cartan subalgebra of $\frakk$,
which is, by definition, a maximal commutative subalgebra of $\frakk$.
$\frakh$ is also a Cartan subalgebra of $\frakg$.

Then it is known that
$(K, \Ad, \fraknp / [\fraknp,\fraknp])$ is a prehomogeneous vector space
\cite{MR0506488},
where $\Ad: K \to \Aut(\fraknp)$ is the adjoint action.
This is called a
{\em prehomogeneous vector space of parabolic type}.
In particular,
when $\frakp$ is a maximal parabolic subalgebra,
and therefore, $\fraknp$ is a commutative subalgebra,
$(K, \Ad, \fraknp)$ 
is called a {\em prehomogeneous vector space of commutative parabolic type}
(see also \cite[Corollaire~4.1.11]{MR2412335}).
\end{definition}
}

Prehomogeneous vector spaces of commutative parabolic type 
are classified as in Table~\ref{tbl:comm-parab-type}.
Type in Table~\ref{tbl:comm-parab-type} means the pair
of the type of the Lie algebra $\frakg$
and the index of the simple root that characterizes
the maximal parabolic subalgebra $\frakp$.

\begin{table}
\includegraphics{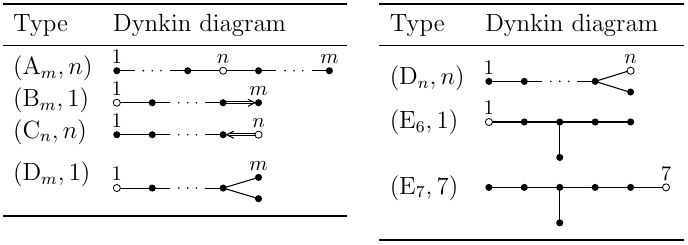}%
\RED{%
\begin{quote}\small
Dynkin diagrams are used
for the classification of complex simple Lie algebras,
and there are types from A to G.
For example, the complex simple Lie algebra of type A$_m$ is $\fraksl_{m+1}$,
and it is the semisimple part of $\frakg$.
White nodes in the above Dynkin diagrams
specify maximal parabolic subalgebra $\frakp$.
The semisimple part of the Levi subalgebra $\frakk$ of $\frakp$
corresponds to the Dynkin diagram obtained by removing the white node 
and its adjacent edges from the Dynkin diagram of $\frakg$.
\end{quote}
}%
\bigskip
\caption{Prehomogeneous vector spaces of commutative parabolic type}
\label{tbl:comm-parab-type}
\end{table}

Among them
relative invariants exist only for types
$({\rm A}_{2n-1}, n)$,
$({\rm B}_m, 1)$,
$({\rm C}_n, n)$,
$({\rm D}_m, 1)$,
$({\rm D}_{2m}, 2m)$ and
{$({\rm E}_7, 7)$,
which are called
{\em regular} 
prehomogeneous vector spaces of commutative parabolic type .
In each type
there exists an irreducible relative invariant,
which is called a {\em basic relative invariant}.
The other relative invariants are powers of 
the basic relative invariant.
We list below the basic relative invariant for each type.
In the list,
Lie algebras, Lie groups and spaces of matrices are defined over $\CC$.
\medskip

\paragraph{\boldmath $({\rm A}_{2n-1}, n)$}
$\frakg = \frakgl_{2n}$,
$K \simeq GL_n \times GL_n$ and
$\fraknp \simeq \Mat_n$.
The action of $K$ on $\fraknp$ is given by
\begin{align*}
(g_1, g_2).X = g_1 X g_2^{-1}
\qquad
((g_1, g_2) \in GL_n \times GL_n, \; X \in \Mat_n),
\end{align*}
and the basic relative invariant is $f(X) = \det X$.

\paragraph{\boldmath $({\rm B}_m, 1)$ and $({\rm D}_m, 1)$}
For type $({\rm B}_m, 1)$,
$\frakg$ is the complex orthogonal Lie algebra $\frako_{2m+1}$,
$K \simeq GL_1 \times O_{2m}$ and
$\fraknp \simeq \CC^{2m}$.
For type $({\rm D}_m, 1)$,
$\frakg = \frako_{2m}$,
$K \simeq GL_1 \times O_{2m-1}$ and
$\fraknp \simeq \CC^{2m-1}$.
Set $n = 2m$ or $2m-1$ as the dimension of $\fraknp$,
then the action of $K$ on $\fraknp$ is given by
\begin{align*}
(a, g).v = a g v
\qquad
(
(a,g) \in GL_1 \times O_n, \; v \in \CC^n
),
\end{align*}
and the basic relative invariant is 
$f = x_1^2 + x_2^2 + \cdots + x_n^2$,
where $x_1, \ldots ,x_n$ are linear coordinate functions on $\CC^n$.

\paragraph{\boldmath $({\rm C}_n, n)$}
$\frakg$ is the complex symplectic Lie algebra $\fraksp_{2n}$
of size $2n$,
$K \simeq GL_n$ and
$\fraknp \simeq \Sym_n$, 
which is the space of symmetric matrices of size $n$.
The action of $K$ on $\fraknp$ is given by
\begin{align*}
g.X = gX \T g
\qquad
(g \in GL_n, \; X \in \Sym_n),
\end{align*}
and the basic relative invariant is $f(X) = \det X$.

\paragraph{\boldmath $({\rm D}_{2m}, 2m)$}
$\frakg = \frako_{4m}$,
$K \simeq GL_{2m}$ and
$\fraknp \simeq \Alt_{2m}$, 
which is the space of alternating matrices of size $2m$.
The action of $K$ on $\fraknp$ is given by
\begin{align*}
g.X = gX \T g
\qquad
(g \in GL_{2m}, \; X \in \Alt_{2m}),
\end{align*}
and the basic relative invariant is the Pfaffian $f(X) = \Pf X$.

\paragraph{\boldmath $({\rm E}_7, 7)$}
In this type
$K$ is isomorphic to the product of the group of type E$_6$
and the group $GL_1$.
The basic relative invariant is a homogeneous polynomial of 
degree three
on 27-dimensional vector space $\fraknp$.
We omit the details
\RED{%
(see \cite[Example~39]{MR430336})%
}%
.
\medskip

If $(K, \Ad, \fraknp)$ is a regular prehomogeneous vector space
of \RED{commutative} parabolic type,
then its contragredient representation
$(K, \Ad, \fraknm)$ with respect to the Killing form
is also a regular prehomogeneous vector space
of \RED{commutative} parabolic type,
and there is a basic relative invariant in $\CC[\fraknm]$,
where $\fraknm \; ( \subset \frakg)$ is the opposite Lie algebra of $\fraknp$.
Take type $({\rm C}_n, n)$ for example,
the action of $(K, \Ad, \fraknm)$
is 
$g.X = \T g^{-1} X g^{-1}$
($g \in GL_n$, $X \in \Sym_n$),
and the basic relative invariant
$\of \in \CC[\fraknm]$ is $\of(X) = \det X$.

\RED{%
In our main theorem (Theorem~\ref{thm:main}),
we deal with regular prehomogeneous vector spaces of commutative parabolic type
described above,
and take the basic relative invariant $\of \in Q = \CC[\fraknm]$.
The ring $R$ of differential operators with constant coefficients,
which acts on $Q$ by differentiation,
is identified with $\CC[\fraknp]$ via the Killing form.
We prove the strong Lefschetz property of $R / \Ann_R(\of^s)$
for any positive integer $s$.
See the next section for details.
}

\section{Main theorem} 
\label{sec:main-thm}
Set $Q = \CC[\fraknm]$.
For $f \in \CC[\fraknp]$
we define a constant\RED{-}coefficient differential operator $f(\partial)$
on $\fraknm$ by
\begin{align} \label{eq:diff-op}
f(\partial) \exp (x, y) = f(x) \exp (x,y)
\qquad
(x \in \fraknp, \; y \in \fraknm),
\end{align}
where $(x,y)$ denotes the Killing form on $\frakg$.
By this identification
we regard $\CC[\fraknp]$ as the ring of partial differential operators
with constant coefficients on $\fraknm$,
and set $R = \CC[\fraknp]$.
The homogeneous component $R_1$ of degree one
can be identified with $\fraknm$.

We have our main theorem,
which is proved in Section~\ref{sec:proof}.

\begin{theorem} \label{thm:main}
Let $\of \in \CC[\fraknm]$ be the basic relative invariant
(see Section~\ref{sec:pv-parab})
of $(K, \Ad, \fraknm)$
which is a regular prehomogeneous vector space of commutative parabolic type.
Set $R = \CC[\fraknp]$,
and $F = \of^s$ for a positive integer $s$.
Then the Artinian Gorenstein algebra $R / \Ann_R(F)$
generated by $F$ has the strong Lefschetz property.

Moreover,
$L \in R_1 \simeq \fraknm$ is a Lefschetz element
if and only if $L$ is in the open $K$-orbit of
the prehomogeneous vector space $(K, \Ad, \fraknm)$
independent of $s$.
\end{theorem}

\begin{example}[Type (C$_n,n$)] \label{ex:main}
We see an example of type $({\rm C_n}, n)$ in Table~\ref{tbl:comm-parab-type}.
Set
\begin{align*}
\frakg &= \fraksp_{2n} =
\left\{
\begin{pmatrix} A&B \\ C&-\T A \end{pmatrix}
\;\Big|\; A \in \frakgl_{n}, \; B, C \in \Sym_n
\right\},
\\
\frakp &= 
\left\{
\begin{pmatrix} A&B \\ 0&-\T A \end{pmatrix}
\;\Big|\; A \in \frakgl_{n}, \; B \in \Sym_n
\right\},
\\
\frakk &= 
\left\{
\begin{pmatrix} A&0 \\ 0&-\T A \end{pmatrix}
\;\Big|\; A \in \frakgl_{n}, 
\right\}
\simeq \frakgl_n,
\\
\fraknp &= 
\left\{
\begin{pmatrix} 0&B \\ 0& 0 \end{pmatrix}
\;\Big|\; B \in \Sym_n
\right\}
\simeq \Sym_n,
\\
\fraknm &= 
\left\{
\begin{pmatrix} 0&0 \\ C& 0 \end{pmatrix}
\;\Big|\; C \in \Sym_n
\right\}
\simeq \Sym_n,
\end{align*}
then
$\frakp$ is a maximal parabolic subalgebra of $\frakg$,
$\frakk$ is a Levi subalgebra of $\frakp$,
and $\fraknp$ is the nilpotent radical of $\frakp$.
The complex Lie group corresponding to $\frakk$ is
\begin{align*}
K &= 
\left\{
\begin{pmatrix} g&0 \\ 0&\T g^{-1} \end{pmatrix}
\;\Big|\; g \in GL_{n}, 
\right\}
\simeq GL_n.
\end{align*}

The adjoint actions of $K$ on $\fraknp$ and $\fraknm$ are given by 
\begin{align*}
\begin{pmatrix} g&0 \\ 0&\T g^{-1} \end{pmatrix}
\begin{pmatrix} 0&B \\ 0& 0 \end{pmatrix}
\begin{pmatrix} g&0 \\ 0&\T g^{-1} \end{pmatrix}^{-1}
&=
\begin{pmatrix} 0 & g B \T g \\ 0 & 0 \end{pmatrix},
\\
\begin{pmatrix} g&0 \\ 0&\T g^{-1} \end{pmatrix}
\begin{pmatrix} 0&0 \\ C& 0 \end{pmatrix}
\begin{pmatrix} g&0 \\ 0&\T g^{-1} \end{pmatrix}^{-1}
&=
\begin{pmatrix} 0&0 \\ \T g^{-1} C g^{-1} & 0 \end{pmatrix}.
\end{align*}
Thus the action of the prehomogeneous vector space
$(K, \Ad, \fraknp)$ 
(resp. $(K, \Ad, \fraknm)$)
is written simply as
$g.B = g B \T g$ 
(resp. $g.C = \T g^{-1} C g^{-1}$)
($g \in GL_n$, $B, C \in \Sym_n$)
as already seen in Section~\ref{sec:pv-parab}.

Since $\of = \det C \in \CC[\fraknm]$
is just multiplied by $\det g^{-2}$ under the action $g.C = \T g^{-1} C g^{-1}$,
$\of$ is a relative invariant.
Moreover, $\of$ is the basic invariant, since it is an irreducible polynomial.

Set $R = \CC[\fraknp]$.
Then Theorem~\ref{thm:main} says that
$R / \Ann_R(\of^s)$ has the strong Lefschetz property
for any positive integer $s$.
A Lefschetz element $L \in \fraknm$ is any symmetric matrix
of rank $n$,
since the open orbit of $(K, \Ad, \fraknm)$
is equal to the set of the matrices of full rank.
\end{example}

\begin{remark}[The set of Lefschetz elements]
Although Theorem~\ref{thm:main}
gives the set of Lefschetz elements completely,
to determine it is very difficult in general,
and there are only a few such examples.

The simplest example of such graded Artinian algebras
that has the strong Lefschetz property is a monomial complete
intersection
$\CC[x_1, x_2, \ldots, x_n] / 
\langle x_1^{a_1}, x_2^{a_2}, \ldots, x_n^{a_n} \rangle$.
Another known example is the case of coinvariant rings of Weyl groups and real
reflection groups
(Maeno-Numata-Wachi \cite{MR2817446}).

In these two known cases the complement of the set of Lefschetz elements is
a union of hyperplanes, but in our main theorem the complement is a
union of hypersurfaces.
\end{remark}

\section{Proof of the main theorem} 
\label{sec:proof}
In the rest of this paper we prove Theorem~\ref{thm:main}.
We do not use Hessians \RED{(see Introduction for Hessians)},
but use the theory of generalized Verma modules of Lie algebras.
In this section we use the notation of Definition~\ref{def:pv-parab},
and suppose that $(K, \Ad, \fraknp)$ is regular.

\subsection{$\ad(\frakk)$-module structure of $\CC[\fraknp]$} 
\label{subsec:str-CC[fraknp]}
\begin{definition}[strongly orthogonal roots]
Let $\Delta$ be the root system of $(\frakg, \frakh)$.
Two roots $\alpha, \beta \in \Delta$ are said to be 
{\em strongly orthogonal}
if $\alpha$ is not proportional to $\beta$,
and neither $\alpha+\beta$ nor $\alpha-\beta$ belongs to $\Delta$.

If $\alpha$ and $\beta$ are strongly orthogonal,
then $\alpha$ is orthogonal to $\beta$,
since
$(\alpha, \beta) < 0$ implies $\alpha - \beta \in \Delta$.
\end{definition}

Let $\alpha_{i_0}$ be the simple root
that characterizes the maximal parabolic subalgebra $\frakp$.
Namely,
$i_0$ is the index of the white circle in Table~\ref{tbl:comm-parab-type}.
Let $\Delta_N^+$ be the set of roots corresponding to $\fraknp$.
We take a sequence $\gamma_1, \gamma_2, \ldots, \gamma_r$
of mutually strongly orthogonal roots in $\Delta_N^+$ as follows.
Set $\gamma_1 = \alpha_{i_0}$.
When we have defined $\gamma_1,\ldots,\gamma_i$,
let $\gamma_{i+1} \in \Delta_N^+$ be the lowest root 
that is strongly orthogonal to all $\gamma_1, \ldots, \gamma_i$
if there exists such a root.

Set 
$\lambda_i = -(\gamma_1 + \gamma_2 + \cdots + \gamma_i)$
($i=1,2,\ldots,r$).
$\lambda_i$ is an integral weight of $\frakg$,
and it can be also considered as that of $\frakk$,
since we can take a common Cartan subalgebra $\frakh$
for $\frakg$ and $\frakk$.
Then we have the structure theorem of $\CC[\fraknp]$
as follows.

\begin{lemma}[Decomposition of {$\CC[\fraknp]$} as an $\ad(\frakk)$-module,
  Schmid {\cite{MR259164}}]
\label{lem:schmid}
$(\frakk, \ad, \CC[\fraknp])$ decomposes 
multiplicity-freely into simple modules as follows.
\begin{align*}
\CC[\fraknp] 
= \bigoplus_{k_1,\ldots,k_r \ge 0} V_{k_1 \lambda_1 + \cdots + k_r \lambda_r},
\end{align*}
where $V_\lambda$ denotes the finite-dimensional simple $\ad(\frakk)$-module
of highest weight $\lambda$.
\qed
\end{lemma}

It is known that
there exist homogeneous polynomials $f_i$ of degree $i$ ($1 \le i \le r$),
and
$V_{k_1 \lambda_1 + \cdots + k_r \lambda_r}$
contains
$f_1^{k_1} f_2^{k_2} \cdots f_r^{k_r}$ ($k_i \ge 0$)
as a highest weight vector. 

Moreover,
$f_r$ in a one-dimensional vector space $V_{\lambda_r}$
is the basic relative invariant
of the regular prehomogeneous vector space $(K, \Ad, \fraknp)$
\cite[Lemma~6.4]{MR1679583}.
Other relative invariants 
$f_r^{k_r}$ ($k_r \ge 1$) are highest weight vectors of 
one-dimensional vector spaces $V_{k_r \lambda_r}$.

Thus $V_{k_1 \lambda_1 + \cdots + k_r \lambda_r}$
consists of homogeneous polynomial of degree
$k_1 + 2k_2 + \cdots + rk_r$.

\begin{example}[Type (C$_n,n$)] \label{ex:main2}
See Example~\ref{ex:main} for the notation.
For Type (C$_n,n$), $r$ is equal to $n$.
Strongly orthogonal roots in $\Delta_N^+$ are
$\gamma_1 = (0,\ldots,0,2)$,
$\gamma_2 = (0,\ldots,0,2,0), \ldots$,
$\gamma_n = (2,0,\ldots,0)$,
and integral weights $\lambda_i$'s are
$\lambda_1 = (0,\ldots,0,-2)$,
$\lambda_2 = (0,\ldots,0,-2,-2), \ldots$,
$\lambda_n = (-2,\ldots,-2,-2)$.

The decomposition into simple $\ad(\frakgl_n)$-modules is as follows:
\begin{multline*}
\CC[\fraknp]
= \CC[x_{ij} \mid 1 \le i \le j \le n]
= \bigoplus_{k_1,\ldots,k_n \ge 0} V_{k_1 \lambda_1 + \cdots + k_n \lambda_n}
\\
= \bigoplus_{0 \le l_1 \le \cdots \le l_n} V_{(-2l_1,-2l_2,\ldots,-2l_n)}
= \bigoplus_{k_1,\ldots,k_n \ge 0} \langle 
f_1^{k_1} f_2^{k_2} \cdots f_n^{k_n}
\rangle_{\ad(\frakgl_n)},
\end{multline*}
where $x_{ij}$ denotes the linear coordinate function on $\fraknp$,
and 
$f_t = \det (x_{ij})_{n-t+1 \le i,j \le n}$.
We use the convention $x_{ij} = x_{ji}$ for $i > j$,
and $\langle f \rangle_{\ad(\frakgl_n)}$ is the $\ad(\frakgl_n)$-submodule
of $\CC[\fraknp]$ generated by $f \in \CC[\fraknp]$.

In the above decomposition
$f_1^{k_1} f_2^{k_2} \cdots f_n^{k_n}$ is the highest weight vector
in $V_{k_1 \lambda_1 + \cdots + k_n \lambda_n}$.
$f_n$ is the basic relative invariant of 
the prehomogeneous vector space $(K, \Ad, \fraknp) = (GL_n, \Ad, \Sym_n)$.
\end{example}

The $\frakk$-module $(\frakk, \ad, \CC[\fraknm])$ is
dual to $(\frakk, \ad, \CC[\fraknp])$,
and has a decomposition into simple $\ad(\frakk)$-modules 
similar to Lemma~\ref{lem:schmid}:
\begin{align*}
\CC[\fraknm]
= \bigoplus_{k_1,\ldots,k_r \ge 0} V_{k_1 \ol_1 + \cdots + k_r \ol_r},
\end{align*}
where $\ol_i = \gamma_{r-i+1} + \gamma_{r-i+1} + \cdots + \gamma_r$
is the highest weight of the contragredient representation
of $V_{\lambda_i}$.
Let $\of_i \in \CC[\fraknm]$
be a highest weight vector of $V_{\ol_i}$.
Then similarly to the case of $(\frakk, \ad, \CC[\fraknp])$,
$\of_1^{k_1} \of_2^{k_2} \cdots \of_r^{k_r} \in \CC[\fraknm]$
is a highest weight vector of $V_{k_1 \ol_1 + \cdots k_r \ol_r}$.

\subsection{$\ad(\frakk)$-module structure of $R/\Ann_R(F)$} 
\label{subsec:str-R/Ann(F)}

\begin{proposition}[Decomposition of {$R / \Ann_R(F)$} as an $\ad(\frakk)$-module]
\label{prop:ann}
Set $R = \CC[\fraknp]$.
The annihilator of a relative invariant $F = \of_r^{\,s} \in \CC[\fraknm]$
of $(K, \Ad, \fraknm)$
for a positive integer $s$ has the following decomposition.
\begin{align*}
\Ann_R(F) =
\bigoplus_{\substack{k_1,\ldots,k_r \ge 0 \\ k_1+\cdots+k_r > s
}}
V_{k_1 \lambda_1 + \cdots + k_r \lambda_r}.
\end{align*}
Therefore the decomposition of 
the Gorenstein algebra generated by $F$ is given by
\begin{align*}
R / \Ann_R(F) \simeq
\bigoplus_{\substack{k_1,\ldots,k_r \ge 0 \\ k_1+\cdots+k_r \le s
}}
V_{k_1 \lambda_1 + \cdots + k_r \lambda_r}.
\end{align*}

\end{proposition}
\begin{proof}
The second decomposition follows from the first one.
We prove the first decomposition. 

Since $F$ is a relative invariant under the action of $\Ad(K)$,
$\Ann_R(F) \subset \CC[\fraknp]$ is an $\ad(\frakk)$-submodule,
and is decomposed into a sum of $V_{k_1 \lambda_1 + \cdots + k_r \lambda_r}$.
Since $V_{k_1 \lambda_1 + \cdots + k_r \lambda_r}$ is irreducible,
$V_{k_1 \lambda_1 + \cdots + k_r \lambda_r} \subset \Ann_R(F)$
if and only if a nonzero polynomial in
$V_{k_1 \lambda_1 + \cdots + k_r \lambda_r}$ annihilates $F$
by using the identification of 
$\CC[\fraknp]$ with differential operators
(Equation~(\ref{eq:diff-op})).

Consider the differentiation
$(f_1^{k_1} f_2^{k_2} \cdots f_r^{k_r})(\partial)\of_r^{\,s}$.
We repeatedly use the formula 
(see \cite[Lemme~5.6]{MR910208}, \cite[Equation~(10.3)]{MR1679583})
\begin{align} \label{eq:(10.3)}
f_{r-i}(\partial) \of_1^{m_1} \of_2^{m_2} \cdots \of_i^{m_i} \of_r^{m+1}
= b_{r-i}(m) \of_1^{m_1} \of_2^{m_2} \cdots \of_{i-1}^{m_{i-1}} \of_i^{m_i+1} \of_r^{m}
\end{align}
(up to nonzero scaling)
for non-negative integers $m_1, \ldots, m_i$ and an integer $m$,
where the polynomial $b_{r-i}(m)$ in $m$ is the $b$-function of $\of_{r-i}$,
and we have
\begin{align*}
(f_1^{k_1} f_2^{k_2} \cdots f_r^{k_r})(\partial)\of_r^s
=
B(s) \of_1^{k_{r-1}} \of_2^{k_{r-2}} \cdots \of_{r-1}^{k_1} \of_r^{s-(k_1+\cdots+k_r)},
\end{align*}
where a polynomial $B(s)$ in $s$ is a product of $b$-functions
of $\of_1, \of_2, \ldots ,\of_r$
($b_{i}$ appears $k_i$ times)
evaluated at $s-k$ ($1 \le k \le k_1 + \cdots + k_r$).
Since it is known that zeros of $b$-functions are negative,
and $m = -1$ is always a (maximum integral) zero of any $b$-function,
$B(s)$ is equal to zero if and only if $s-k$ can be $-1$.
Namely, $B(s) = 0$ if and only if $s - (k_1 + \cdots + k_r) < 0$

Thus we have proved that
$V_{k_1 \lambda_1 + \cdots + k_r \lambda_r} \subset \Ann_R(F)$
if and only if 
$k_1 + k_2 + \cdots + k_r > s$.
\end{proof}

\begin{remark}[Generating set of $\Ann_R(F)$]
When $F = \of_r^s$ for a positive integer $s$,
$\Ann_R(F)$ is generated by $V_{(s+1)\lambda_1}$ 
as an ideal of $R = \CC[\fraknp]$.

Indeed,
$\Ann_R(F)$ is generated by
$V_{k_1 \lambda_1 + \cdots + k_r \lambda_r}$ with $k_1 + k_2 + \cdots + k_r > s$
by Proposition~\ref{prop:ann},
and this condition is weakened to $k_1 + k_2 + \cdots + k_r = s+1$,
since the highest weight vector of $V_{k_1 \lambda_1 + \cdots + k_r \lambda_r}$
is $f_1^{k_1} f_2^{k_2} \cdots f_r^{k_r}$.
By using Equation~(\ref{eq:(10.3)})
we can prove that
\begin{align*}
\CC[\fraknp]_1 
V_{k_1\lambda_1 + \cdots + k_{i-2}\lambda_{i-2} + 
  (k_{i-1}+1) \lambda_{i-1} + k_i \lambda_i}
\supset 
V_{k_1\lambda_1 + \cdots + k_{i-1}\lambda_{i-1} + (k_{i}+1) \lambda_{i}}.
\end{align*}
For the proof of the above formula
we need to consider `lower-rank version' of Equation~(\ref{eq:(10.3)}),
but we omit the details
(see 
\cite[Section~8]{MR1679583},
\cite[Section~2]{MR834108}
for the `lower-rank' setting).
Therefore,
by repeated use of the above formula,
it follows that
$\CC[\fraknp] V_{(s+1)\lambda_1} \supset V_{k_1 \lambda_1 + \cdots + k_r \lambda_r}$
whenever $k_1 + k_2 + \cdots + k_r = s+1$,
and we have proved 
$\Ann_R(F)$ is generated by $V_{(s+1)\lambda_1}$ 
as an ideal of $R$.

Moreover,
we can conclude that 
$\Ann_R(F)$ is generated by $f_1^{s+1}$
as an $\ad(\frakk)$-stable ideal of $R$.
For example, in Type (C$_n,n$)
(see Example~\ref{ex:main2} for the notation)
$\Ann_R(\of_n^{s})$ is generated by $x_{nn}^{s+1}$
as an $\ad(\frakgl_n)$-stable ideal of $R = \CC[\fraknp] = \CC[\Sym_n]$.
\end{remark}

\begin{example}[Narayana numbers]
In Type (C$_n,n$) (see Example~\ref{ex:main2} for the notation),
when we consider the basic relative invariant $F = \of_n \in \CC[\fraknm]$,
it follows from Proposition~\ref{prop:ann} that
\begin{align*}
R / \Ann_R(\of_n) 
&\simeq
V_0 + V_{\lambda_1} + V_{\lambda_2} + \cdots + V_{\lambda_n}
\\
&= 
V_{(0,\ldots,0)} + V_{(0,\ldots,0,-2)} + V_{(0,\ldots,0,-2,-2)} + \cdots +
V_{(-2,\ldots,-2)}.
\end{align*}
Since this decomposition coincides with homogeneous decomposition
as a graded algebra,
we can compute the Hilbert function of $R / \Ann_R(\of_n)$,
which is written as a sequence of dimensions of homogeneous components
in this paper,
using the Weyl dimension formula for irreducible representations
of $\frakgl_n$:
\begin{align*}
&\Hilb(R / \Ann_R(\of_n))
\\&\quad=
\left(
\tfrac{1}{n+1}\tbinom{n+1}{1}\tbinom{n+1}{0}, \;
\tfrac{1}{n+1}\tbinom{n+1}{2}\tbinom{n+1}{1}, \;
\ldots,
\tfrac{1}{n+1}\tbinom{n+1}{n+1}\tbinom{n+1}{n}
\right).
\end{align*}
This sequence consists of {\em Narayana numbers},
which are originated in combinatorics,
and defined as
$N(n,k) = \frac{1}{n} \binom{n}{k} \binom{n}{k-1}$.
\end{example}

\subsection{Generalized Verma modules} 
Define 
$\frakg$,
$\frakp$,
$\frakk$,
$K$ and
$\fraknp$
as in Definition~\ref{def:pv-parab}.
In addition, let $\fraknm$ be the opposite of $\fraknp$.
Suppose $(K, \Ad, \fraknp)$ is regular.

Let $(\frakp, \mu, \CC_\mu)$ be 
a one-dimensional representation of $\frakp$ ($\CC_\mu = \CC$),
and set
$M(\mu) = U(\frakg) \otimes_{U(\frakp)} \CC_\mu$,
which is called a {\em generalized Verma module} of $\frakg$
induced from $\mu$.
Since $(\frakp, \mu, \CC_\mu)$ is one-dimensional,
$\mu$ is a multiple of $\varpi_{i_0}$,
which is the fundamental weight of $\frakg$ corresponding to the simple root
$\alpha_{i_0}$
that characterizes the maximal parabolic subalgebra $\frakp$.
In addition, $\lambda_r$ is also a multiple of $\varpi_{i_0}$,
since $(\frakk, \ad, V_{\lambda_r})$ is one-dimensional.
Indeed, $\lambda_r = -2 \varpi_{i_0}$
\cite[Lemma~6.4]{MR1679583}.

Then it follows from $\frakg = \fraknm \oplus \frakp$ that
\begin{align*}
M(\mu)
=
U(\frakg) \otimes_{U(\frakp)} \CC_\mu
=
U(\fraknm) U(\frakp) \otimes_{U(\frakp)} \CC_\mu
\simeq
U(\fraknm)
\end{align*}
as $\CC$-vector spaces.
Since $\fraknm$ is a commutative Lie algebra,
we have
$M(\mu) \simeq S(\fraknm) \simeq \CC[\fraknp]$
as vector spaces,
where $S(\fraknm)$ is the symmetric algebra of $\fraknm$,
and the second isomorphism is by
$\fraknm \simeq (\fraknp)^*$ via the Killing form on $\frakg$.
Thus we have the action of $\frakg$ on $\CC[\fraknp]$,
and we denote this representation by
$(\frakg, \psi_\mu, \CC[\fraknp])$.

The explicit form of the action of $(\frakg, \psi_\mu, \CC[\fraknp])$
is given in the following lemma.
The action of $\fraknp$ is, in fact,
a differential operator of second order with polynomial coefficients,
though
we do not need it, and omit the explicit form.

\begin{lemma}[Actions of generalized Verma modules, {\cite[Lemma~3.2]{MR1679583}}] \label{lem:psi_mu}
The action of $\fraknm$ and $\frakk$ 
under the representation $(\frakg, \psi_\mu, \CC[\fraknp])$
is given as follows:
\begin{align*}
\psi_\mu(X) &= \times X
&& (X \in \fraknm),
\\
\psi_\mu(X) &= \ad(X) + \mu(X),
&& (X \in \frakk),
\end{align*}
where 
$\times X$ denotes the multiplication map on $\CC[\fraknp]$ by 
a linear polynomial $X \in \fraknm \simeq (\fraknp)^* \simeq \CC[\fraknp]_1$,
and $\mu(X)$ denotes the multiplication map by 
a scalar $\mu(X)$.
\end{lemma}
\begin{proof}
Let $f \in \CC[\fraknp]$.
We regard $f$ as an element in $S(\fraknm)$, 
and $f \otimes 1_\mu \in M(\lambda)$,
where $1_\mu = 1 \in \CC_\mu$ is the basis of $\CC_\mu$.

If $X \in \fraknm$, then
$X(f \otimes 1_\mu) = X f \otimes 1_\mu$,
and $X f \in S(\fraknm)$.
Therefore $\psi_\mu(X)f = X f$.

If $X \in \frakk$, then
\begin{align*}
X(f \otimes 1_\mu) = X f \otimes 1_\mu
= (f X + [X,f]) \otimes 1_\mu
&= f \otimes \mu(X) 1_\mu + \ad(X)f \otimes 1_\mu
\\&= (\mu(X)f + \ad(X)f) \otimes 1_\mu.
\end{align*}
Therefore $\psi_\mu(X)f = \ad(X)f + \mu(X)f$.
\end{proof}

\begin{remark}[Decomposition of {$\CC[\fraknp]$} as a $\psi_\mu(\frakk)$-module]
The decomposition of $\CC[\fraknp]$ into simple $\psi_\mu(\frakk)$-modules
coincides with that as $\ad(\frakk)$-modules,
since these two actions differ only by the constant multiplication by $\mu(X)$.
But highest weights of irreducible components change by the constant $\mu(X)$.
\end{remark}

\subsection{Maximal submodules of $M(\mu)$} 
We continue to use the notation of the previous subsection.

It is known that there exists a unique maximal submodule of $M(\mu)$,
and denote the submodule by $Y_{\mu}$.
If we regard $Y_{\mu}$ as an $\ad(\frakk)$-submodule of $\CC]\fraknp]$,
$Y_{\mu}$ must decompose into a sum of simple modules $V_\lambda$.

\begin{lemma}[Maximal submodules of $M(\mu)$, {\cite[Lemma~10.3, Proposition~10.7]{MR1679583}}]
\label{lem:q-mu}
Let $\mu = s \varpi_{i_0}$ ($s \in \CC$) be a one-dimensional representation of $\frakp$,
and consider the generalized Verma module $M(\mu)$.
For 
$\lambda = k_1 \lambda_1 + \cdots + k_r \lambda_r$,
$V_\lambda \subset Y_\mu$ 
if and only if $q_\mu(\lambda) = 0$,
where the polynomial $q_\mu(\lambda)$ in $k_1, k_2, \ldots, k_r$ is defined as
\begin{align*}
q_\mu(\lambda) = 
\prod_{i=0}^{r-1}
\prod_{l=0}^{k_{i+1}+\cdots+k_r-1}
( \frac{id}{2} + s - l),
\end{align*}
where
$d = 1, 2, 4, 2m-3, 2m-4$
for
$({\rm C}_n, n)$,
$({\rm A}_{2n-1}, n)$,
$({\rm E}_7, 7)$,
$({\rm B}_m, 1)$,
$({\rm D}_m, 1)$,
respectively.
\qed
\end{lemma}

\begin{proposition}[$\Ann_R(F)$ is a submodule of $M(\mu)$]
\label{prop:Ann(F)-is-submodule}
Let $s$ be a positive integer, and $\mu = s \varpi_{i_0}$.
Then the maximal submodule $Y_\mu$ of
the generalized Verma module $M(\mu) \simeq (\frakg, \psi_\mu, \CC[\fraknp])$
is equal to $\Ann_R(\of_r^s) \subset \CC[\fraknp]$.

Therefore $\frakg$ acts on $R / \Ann_R(\of_r^s)$ via $\psi_\mu$.
\end{proposition}
\begin{proof}
By Lemma~\ref{lem:q-mu}
an irreducible component $V_\lambda$
($\lambda = k_1 \lambda_1 + \cdots + k_r \lambda_r$)
of $(\frakg, \psi_\mu, \CC[\fraknp])$
is contained in $Y_\mu$
if and only if
$id/2 + s - l = 0$ for some
$i = 0,1,\ldots,r-1$ and $l = 0,1,\ldots,k_{i+1}+k_{i+2}+\cdots+k_r-1$.
This is equivalent to 
that the minimum of $id/2 - l$ is less than or equal to $-s$.
$id/2$ takes the minimum value when $i=0$,
and 
$-l$ takes the minimum value when $i=0$ and $l = k_1+k_2+\cdots+k_r-1$.
Thus the minimum of $id/2 - l$ is equal to $-(k_1+k_2+\cdots+k_r-1)$.
Therefore $V_\lambda \subset Y_\mu$
if and only if
$ - (k_1+k_2+\cdots+k_r-1) \le -s$,
that is,
$ k_1+k_2+\cdots+k_r > s$.

On the other hand 
$V_\lambda \subset \Ann_R(\of_r^s)$
if and only if
$ k_1+k_2+\cdots+k_r > s$
by Proposition~\ref{prop:ann}.
Thus we have proved 
$\Ann_R(\of_r^s) = Y_\mu$.
\end{proof}

\subsection{Proof of the main theorem} 

The following lemma is 
essentially the same as \cite[Theorem~3.32]{MR3112920},
where the multiplication map by $L$ corresponds to $X \in \fraksl_2$,
but $\times L$ corresponds to $Y \in \fraksl_2$ in our lemma
for the proof of Theorem~\ref{thm:main}.

\begin{lemma}[Condition for the strong Lefschetz property]
\label{lem:key}
Let $I$ be a homogeneous ideal of $R = \CC[x_1,x_2,\ldots,x_n]$.
Let $\fraksl_2 = \CC X + \CC Y + \CC H$
be the complex simple Lie algebra,
where
$[X,Y] = H$,
$[H,X] = 2X$, and
$[H,Y] = -2Y$.

When $A = R/I$ is a graded Artinian algebra
with a symmetric Hilbert function,
the following two conditions are equivalent:
\begin{enumerate}[(1)]
  \item
  $A = R/I$ has the strong Lefschetz property,
  and $L \in A_1$ is a Lefschetz element.
  \item
  There exists an action of $\fraksl_2$ on $A$
  such that
  \begin{enumerate}[(a)]
    \item 
    The weight space decomposition of $A$ coincides with 
    the homogeneous decomposition of $A$, and
    \item 
    The action of $Y \in \fraksl_2$ on $A$
    coincides with
    the multiplication map by $L \in A_1$.
    \qed
  \end{enumerate}
\end{enumerate}
\end{lemma}

Set $R = \CC[\fraknp]$, and $F = \of_r^s \in \CC[\fraknm]$, 
where $s$ is a positive integer.
Set $\mu = s\varpi_{i_0}$ 
so that $\frakg$ acts on $R / \Ann_R(F)$ through $\psi_\mu$
(Proposition~\ref{prop:Ann(F)-is-submodule}). 
To prove the strong Lefschetz property of $R / \Ann_R(F)$,
we will take an $\fraksl_2$-triple $X, Y, H \in \frakg$
so that the action of $\fraksl_2$ on $R / \Ann_R(F)$ via $\psi_\mu$
satisfies the condition (2) of Lemma~\ref{lem:key}.

First, $H \in \frakh$ is uniquely determined.
Indeed, by the condition (2) (a) of Lemma~\ref{lem:key},
the action of $\ad(H)$ on $\fraknm \; (\simeq \CC[\fraknp]_1)$
should be the multiplication by $-2$,
since the eigenspaces of $\psi_\mu(H) = \ad(H) + \mu(H)$
on $R / \Ann_R(F)$ coincide with the homogeneous spaces of $R / \Ann_R(F)$,
and the eigenvalues must decrease by two when the degrees of homogeneous spaces increase by one.
Therefore $H$ is the unique element in the one-dimensional center of $\frakk$,
which is contained in the Cartan subalgebra $\frakh$, satisfying
\begin{align*}
& \alpha_{i_0}(H) = 2,
\\
& \alpha(H) = 0 \quad (\text{for any simple root $\alpha$ other than $\alpha_{i_0}$}).
\end{align*}

The existence of $X \in \fraknp$ and $Y \in \fraknm$ such that $[X, Y] = H$
is classified using weighted Dynkin diagrams
(the Dynkin-Kostant classification, \cite[Theorem~3.5.4]{MR1251060}).
But in our setting
we can find $X$ and $Y$ without argument about the classification.
Namely, we can take $X$ and $Y$ as
\begin{align*}
X = X_{\gamma_1} + X_{\gamma_2} + \cdots + X_{\gamma_r} \in \fraknp,
\qquad
Y = Y_{\gamma_1} + Y_{\gamma_2} + \cdots + Y_{\gamma_r} \in \fraknm,
\end{align*}
where 
$X_{\gamma_i} \in \fraknp$ and $Y_{\gamma_i} \in \fraknm$ are root vectors
corresponding to $\pm\gamma_i$
such that $X, Y, H$ forms an $\fraksl_2$-triple.

Then the action of $\psi_\mu(Y)$ on $R / \Ann_R(F)$ is the multiplication by
$Y \in \fraknm \simeq \CC[\fraknp]_1$ (Lemma~\ref{lem:psi_mu}),
and the weight space decomposition with respect to $\psi_\mu(H)$
on $R / \Ann_R(F)$ coincides with the homogeneous decomposition.
Therefore it follows from Lemma~\ref{lem:key} 
that $R / \Ann_R(F)$ has the strong Lefschetz property,
and we have proved the first paragraph of Theorem~\ref{thm:main}.

A linear coordinate change by the action of $\Ad(K)$ on $\CC[\fraknp]$
causes an automorphism on $\CC[\fraknp] / \Ann_R(F)$,
and clearly preserves the strong Lefschetz property.
Therefore every element of $\Ad(K)$-orbit on $\fraknm$
through $Y \in \fraknm$ is a Lefschetz element.
Conversely,
representatives of orbits of lower dimensions are
$Y_{\gamma_1} + Y_{\gamma_2} + \cdots + Y_{\gamma_i} 
\in \fraknm$ ($0 \le i < r$)
\cite[Th\'eor\`eme~2.8]{MR834108},
and this element never gives an $\fraksl_2$-triple containing $H$.
This means that the set of Lefschetz element is equal to
the open $\Ad(K)$-orbit on $\fraknm$.

\begin{example}[Type (C$_n,n$)]
In the case of type $({\rm C}_n, n)$
(see Example~\ref{ex:main2} for the notation),
$\fraksl_2 \subset \frakg = \fraksp_{2n}$ is given by
\begin{align*}
H &= \begin{pmatrix} 1_n&0 \\ 0&-1_n \end{pmatrix},
&
X &= \begin{pmatrix} 0 & 1_n \\ 0&0 \end{pmatrix},
&
Y &= \begin{pmatrix} 0 & 0 \\ 1_n&0 \end{pmatrix},
\end{align*}
where $1_n$ denotes the identity matrix of size $n$.

The set of Lefschetz elements is the set of full-rank matrices
in $\fraknm$, which is the $GL_n$-orbit through $Y$.
In particular, $Y \in \fraknm$ is a Lefschetz element, 
and, in a form of polynomial, it is
$x_{11} + x_{22} + \cdots + x_{nn} \in \CC[\fraknp]$.
\end{example}

\bibliographystyle{plain}
\bibliography{lef-gor-relinv.bib}

\end{document}